\documentclass[11pt]{article}
\usepackage{amssymb}
\usepackage{amsmath}
\usepackage{amscd}
\usepackage{amsbsy}
\usepackage{amsfonts}
\usepackage{theorem}
\usepackage{hyperref}

\input xy
\xyoption{all}

\newcommand{\zz}{{\Bbb Z}}

\newcommand{\nn}{{\Bbb N}}

\newcommand{\pp}{{\Bbb P}}

\newcommand{\ddim}{\operatorname{dim}}

\newcommand{\op}[1]{\operatorname{#1}}

\newcommand{\ffi}{\varphi}

\newcommand{\eps}{\varepsilon}

\newcommand{\row}{\rightarrow}

\newcommand{\low}{\leftarrow}

\renewcommand{\leq}{\leqslant}

\newcommand{\nichego}[1]{}

\newcommand{\ov}[1]{\overline{#1}}

\newcommand{\wt}[1]{\widetilde{#1}}

\newcommand{\smk}{{\mathbf{Sm_k}}}

\newcommand{\CH}{\op{CH}}

\newcommand{\Qed}{\hfill$\square$\smallskip}
\newcommand{\Red}{\hfill$\triangle$\smallskip}

\newenvironment{proof}{\noindent{\it Proof}:}{\vskip 5mm}

\newtheorem{prop}{Proposition}[section]{\bf}{\it}
\newtheorem{thm}[prop]{Theorem}{\bf}{\it}
{\bf}{\it}
{\bf}{\it}
\newtheorem{defi}[prop]{Definition}{\bf}{\it}
{\bf}{\it}
{\bf}{\it}
{\bf}{\it}
\newtheorem{rem}[prop]{Remark}{\bf}{}
{\bf}{\it}

{\bf}{\it}
{\bf}{\it}
\newtheorem{cor}[prop]{Corollary}{\bf}{\it}
{\bf}{\it}

\begin{document}

\title{On the numerical triviality of $BP$-cycles}
\author{Alexander Vishik\footnote{School of Mathematical Sciences, University
of Nottingham}}
\date{}
\maketitle

\begin{abstract}
We show that, in the case of a prime $2$, the numerical triviality of $BP^*$-cycles modulo various powers of the (augmentation) invariant ideal $I(\infty)$ is controlled by pure symbols in $K^M_*/2$ over the flexible closure of the base field.
\end{abstract}

\section{Introduction}

In \cite{Mi} Milnor introduced his K-theory taking inspiration from the properties of Pfister forms \cite{Pf}. This graded ring appeared to be the diagonal part of motivic cohomology 
of a point. The Pfister forms correspond to special elements of the Milnor's K-theory (mod 2) called {\it pure symbols} and the respective Pfister quadrics are {\it norm-varieties} for these symbols, in the sense, that the triviality of the symbol (over field extensions) is equivalent to the $2$-isotropy of the quadric. This interplay between pure symbols and norm-varieties was at the center of the proofs of all the cases of the Milnor and Bloch-Kato conjectures by Merkurjev, Suslin, Rost and Voevodsky
\cite{MeNRS2,MS,RoNVAC,VoMil,VoZl}.  

In \cite[Theorem 1.1]{VPS} it was shown that any projective variety is a norm-variety for an appropriate pure symbol (mod 2). The only thing, this symbol is defined over some purely transcendental extension, rather than the base field itself. That is, the $2$-isotropy of varieties is controlled by pure symbols in the Milnor's K-theory $K^M_*(\wt{k})/2$ of the {\it flexible closure} $\wt{k}=k(\pp^{\infty})$.
Combining it with the main result of \cite{INCHKm} permits to show that the same holds for
numerical triviality of cycles with $\zz/2$-coefficients - it is also controlled by
similar pure symbols. Note that the anisotropy of a particular variety is exactly the numerical triviality of its generic cycle. In this sense, the results of \cite{INCHKm}
permit to reduce the numerical triviality of an arbitrary cycle (on some variety) to the
numerical triviality of the generic one (on another variety). But this works only for
sufficiently good oriented cohomology theories. The \cite[Theorem 4.8]{INCHKm}
describes an important class of such theories. These are quotients $BP^*/I$ of the 
$BP$-theory by invariant ideals. Examples of such theories are $BP^*/I(\infty)=\CH^*/p=K(\infty)^*$ and $P(m)^*=BP^*/I(m)$. The \cite[Theorem 4.8]{INCHKm} claims
that the theories $Q^*$ of the mentioned shape satisfy the {\it Isotropy Conjecture},
that is, their {\it isotropic} $Q^*_{iso}$ and {\it numerical} $Q^*_{Num}$ versions
coincide over {\it flexible} fields. In practical terms, that means that any 
numerically-trivial class, over some purely transcendental extension, comes as a push-forward from some anisotropic variety. But a careful analysis of the proof of this theorem shows that, in addition, the mentioned (anisotropic) variety controls the numerical triviality of the cycle - this is done in the Theorem \ref{nt-an} below.
This reduces numerical triviality to isotropy - see Remark \ref{rem-nt-an}.

The symmetric power $S^p$ operation permits to reduce the $p^r$-isotropy to the $p$-isotropy. In particular, (recalling \cite[Theorem 1.1]{VPS}) this shows that the $2^r$-isotropy of projective varieties is also controlled by pure symbols in
$K^M_*(\wt{k})/2$. We would be able to claim the same about the numerical triviality
of cycles with $\zz/2^r$-coefficients, if we would be able to reduce it to the $2^r$-isotropy of varieties. Unfortunately, the theory $\CH^*/p^r$, for $r>1$, doesn't satisfy
the conditions of \cite[Theorem 4.8]{INCHKm}, as the ideal $(p^r,v_1,v_2,\ldots)$ is not invariant. And it is not occasional - it was shown in \cite{contr-IN} that the {\it Isotropy Conjecture} fails for such theories. But there is another theory $BP^*/I(\infty)^r$ which has the identical isotropy properties: varieties are anisotropic in both senses simultaneosly. The advantage of it is that it satisfies the conditions of \cite[Theorem 4.8]{INCHKm}, in particular, the {\it Isotropy Conjecture} holds for it.
Thus, $BP^*/I(\infty)^r$ may be considered as a "regular" substitute for $\CH^*/p^r$. As a result, we may reduce the numerical triviality of classes in $BP^*/I(\infty)^r$ to the $BP/I(\infty)^r$-anisotropy of varieties, which is equivalent to 
the $\CH/p^r$-anisotropy and so, may be reduced to the $\CH/p$-anisotropy (using symmetric powers). Finally, in the case of the prime $2$, the latter is controlled by
pure symbols from $K^M_*(\wt{k})/2$ and so, the numerical triviality of the original 
$BP^*/I(\infty)^r$-classes is controlled by such symbols - see Theorem \ref{ntBP-ps}.
Moreover, we can do it in a coherent way, so that the symbols for different $r$ would
divide each other appropriately.

\section{Isotropic and numerical equivalence}

Let $A^*$ be an oriented cohomology theory in the sense of 
\cite[Definition 2.1]{SU} (so, with the {\it localisation} axiom).
Following \cite{INCHKm}, we may introduce the notion of
$A^*$-isotropy.

\begin{defi}
Let $X\stackrel{\pi}{\row}\op{Spec}(k)$ be a projective variety
over $k$. We say that $X$ is $A$-anisotropic, if the push-forward map
$\pi_*:A_*(X)\row A_*(\op{Spec}(k))=A$ is zero.
\end{defi}

Then we can introduce the {\it isotropic} version of our theory $A^*$.
We close the set of $A^*$-classes on anisotropic varieties under push-forwards (it is automatically closed under pull-backs). Moding these out, we get $A^*_{iso}$.

\begin{defi}
Let $X$ be projective, we say that a class $x\in A^*(X)$ is
$A$-anisotropic, if it comes as a push-forward of some element from
$A$-anisotropic variety.

We define: $A^*_{iso}(X)=A^*(X)/(\,\text{anisotropic classes}\,)$.
\end{defi}

Using \cite[Example 4.1]{Iso} (cf. \cite[Example 4.6]{RNCT}) this can be
extended to an oriented cohomology theory on $\smk$.

If $X$ is smooth projective, we have the natural "degree" pairing
$$
\langle -,-\rangle:\, A^*(X)\times A^*(X)\row A; \hspace{5mm}
\langle u,v\rangle=\pi_*(u\cdot v).
$$

\begin{defi}
We say that $u\in A^*(X)$ is numerically trivial, if the pairing
$\langle u,-\rangle: A^*(X)\row A$ is zero.

We define: $A^*_{Num}(X)=A^*(X)/(\,\text{numerically trivial classes}\,)$.
\end{defi}

Again, this extends to an oriented cohomology theory on $\smk$.

Since the $A$-degree pairing on $A$-anisotropic varieties is zero,
from the projection formula, any $A$-anisotropic class is $A$-numerically trivial and we get natural surjections of oriented theories:
$$
A^*\twoheadrightarrow A^*_{iso}\twoheadrightarrow A^*_{Num}.
$$

The natural question arises: in which situations, can we expect the {\it isotropic} version of the theory to coinside with the {\it numerical} one?
As the numerical version of the theory doesn't change under purely transcendental extensions and the isotropic one may easily change there, it is natural to ask it for {\it flexible} ground fields \cite[Definition 1.1]{Iso}, that is, purely transcendental extensions of infinite transcendence degree of some other fields.
If these versions coinside, we say that the {\it Isotropy Conjecture} holds for $A^*$.
One obvious necessary condition for this is that $A$ must be torsion, i.e.
there should exist $n\in\nn$, s.t. $n\cdot A=0$ - see \cite[Remark 2.4]{INCHKm}. But not all torsion theories are such. For example, it was
shown in \cite{contr-IN} that for $\CH^*/p^r$, with $r>1$, the numerical and isotropic versions are different.  

But we know a large natural class of (free) theories for which the Isotropy Conjecture holds. In \cite[Theorem 4.8]{INCHKm} it was proven that it is so for
any theory of the type $Q^*=BP^*/I$, where $I\subset BP$ is a non-zero ideal 
invariant under Landweber-Novikov operations.
In particular, it holds for $\CH^*/p=BP^*/I(\infty)=K(\infty)$. By \cite[Theorem 4.17]{INCHKm}, it also holds for all other 
Morava K-theories $K(n)$, $1\leq n\leq\infty$. 

Observe, that anisotropic classes may easily stop to be anisotropic over field extensions and the same may happen to numerically trivial ones. 
But a careful analysis of the proof of \cite[Theorem 4.8]{INCHKm}
reveals that, not only any given numerically trivial class comes via push-foward from a certain anisotropic variety, but also the numerical triviality  
of it over field extensions is equivalent to the anisotropy of this variety there.

\begin{thm}
 \label{nt-an}
Let $Q^*=BP^*/I$ be a free oriented cohomology theory, where $I\subset BP$ is a non-zero ideal invariant under Landweber-Novikov operations.  Let $X$ be a smooth projective variety and $u\in Q^*(X)$
be a numerically trivial element. Then, over some purely transcendental extension $k(\pp)$ of $k$, there exists a map 
$\ffi:X'\row X_{k(\pp)}$, where $X'/k(\pp)$ is smooth projective and $Q$-anisotropic,
such that $u_{k(\pp)}=\ffi_*(u')$, for some $u'\in Q^*(X')$.
Moreover, for any invariant ideal $J$ with $I\subset J\subset BP$ and 
$\ov{Q}^*=BP^*/J$, and any field extension $E/k$, $\ov{u}_E$ is $\ov{Q}$-numerically 
trivial if and only if $X'_{E(\pp)}$ is $\ov{Q}$-anisotropic.
\end{thm}

\begin{proof}
Only the last statement requires a proof, since the rest is contained in
\cite[Theorem 4.8]{INCHKm}, but we will need to recall the main steps of the proof of the mentioned theorem.

Let us fix the extension $E/k$ and check that the construction of loc. cit. produces a variety $X'$ which stays $Q$-anisotropic over $E(\pp)$ as long as the class $u$ stays $Q$-numerically trivial over $E$ (clearly, if $u_E$ is not numerically trivial anymore, then $X'_{E(\pp)}$ will not be anisotropic, so the other case is obvious).

\noindent \underline{{\bf Step 1}}:
We start with $u\in Q^*(X)$, lift it to $\Omega^*_{\zz_{(p)}}(X)$, apply the multiplicative projector $\rho$ defining the $BP^*$-theory and multiply by an appropriate integer equal to $1$ in $Q$ to obtain a class
$[Y\stackrel{y}{\row}X]\in\Omega^*(X)$, applications of all Landweber-Novikov operations to
which are $Q$-numerically trivial - see \cite[Proposition 4.3]{INCHKm}.
All these manipulations commute with the passage to $E$ and the property will still hold for $y_E$ as long as $u_E$ is $Q$-numerically trivial (by \cite[Proposition 4.1]{INCHKm}). 

\noindent \underline{{\bf Step 2}}:
Our class is now represented by $u=[Y\stackrel{y}{\row}X]\in\Omega^*(X)$, where all the polynomials in the Chern classes of $T_Y$ are numerically trivial on $X$, as these are the Landweber-Novikov operations applied to $y$ (and the same is true for $y_E$). Pick some very ample line bundle $L$ on $Y$. Then, for sufficiently large $r$, the bundle $T_Y\otimes L^{p^r}$ will define a regular embedding $f:Y\hookrightarrow Gr(d,N)$ into some Grassmannian, where $d=\ddim(Y)$ and $T_Y\otimes L^{p^r}=f^*(Tav)$ is the pull-back of the tautological bundle. Moreover, if $r$ is large, then $T_Y$ and $T_Y\otimes L^{p^r}$ will have the same
$Q$-Chern classes - see \cite[Lemma 4.9]{INCHKm} (we say that these vector bundles are {\it $Q$-equivalent}). We may substitute $u$ by $v=[Y\stackrel{g}{\row}X\times Gr(d,N)]$,
where $g=(y,f)$, since $Q^*(X\times Gr(d,N))$ is generated as a
$Q^*(X)$-algebra by the Chern classes of the tautological vector bundle,
which restricted to $Y$ give Chern classes of $T_Y$, polynomials in which are numerically trivial (on $X$), so $v$ is $Q$-numerically trivial
(and the same holds over $E$). Also, $u$ is the push-forward of $v$, so
if $v$ comes as a push-forward from some anisotropic variety, then the same variety may be used for $u$. Denoting our new class still
as $u=[Y\stackrel{y}{\row}X]$, we have that it is now represented by
a regular embedding whose normal bundle is {\it $Q$-equivalent} to the pull-back of a virtual
vector bundle $y^*([V]-[U])$, where $V=T_X$ and $\ddim(U)=\ddim(Y)$. Substituting $[Y\stackrel{y}{\row}X]$ by $[Y\stackrel{y}{\row}X\row\pp_X(O\oplus U)]$ (where both elements are obtained from each other by push-forwards), we may assume
that $N_{y}$ is $Q$-equivalent to the pull-back of some vector
bundle $V$ on $X$ of the same dimension.  Our class is still numerically trivial over $k$ and $E$ and if we may realise it as a push-forward from a variety which stays anisotropic over $E$, then the same variety will realise
the original class.

\noindent \underline{{\bf Step 3}}:
Now $u=[Y\stackrel{y}{\row}X]$ is the class of a regular embedding,
where $N_{y}$ is $Q$-equivalent to $y^*(V)$ for some vector bundle
$V$ on $X$ with $\ddim(N_y)=\ddim(V)$. As before, $u$ is 
$Q$-numerically trivial and so is $u_E$. Let $\eps:Fl(V)\row X$ be the variety of complete flags on $V$. Then $\eps^*(u)$ and $\eps^*(u)_E$
are still $Q$-numerically trivial, and $u$ may be obtained from $\eps^*(u)$ via multiplication by certain Chern classes and push-forwards.
So, if $\eps^*(u)$ is realised by a certain variety which stays anisotropic over $E$, then $u$ will be realised by some subvariety of it which must then stay anisotropic over $E$ as well. Thus, substituting $u$ by $\eps^*(u)$, we may assume that $[V]=[L_1]+\ldots+[L_n]\in K_0(X)$
is the sum of classes of line bundles. 

\noindent \underline{{\bf Step 4}}:
Here we apply the deformation to the normal cone construction.\\
Let $W=Bl_{X\times\pp^1}(Y\times 0)\stackrel{\pi}{\row}X\otimes\pp^1$.
Let $\ov{u}=[Y\times 1]\in Q^*(X\times\pp^1)$ and
$c_1^Q(L_i)=a_i\in Q^*(X)$. Then, 
$-\pi^*(\ov{u})=\xi(\xi+_Qa_1)\cdot\ldots\cdot(\xi+_Qa_n)$, where $\xi=c_1(O(1))$ -
see \cite[Lemma 4.10]{INCHKm}. Since $u$ was numerically trivial,
so is $-\pi^*(\ov{u})$ (and the same holds over $E$). If $-\pi^*(\ov{u})$ may be realised as a push-forward from some variety which stays anisotropic over $E$, then the same variety may be used to realise $u$,
since $-u$ is obtained from $-\pi^*(\ov{u})$ via push-forward.
Thus, we may assume that $u$ is a product of 1-st $Q$-Chern classes of line bundles. Multiplying the respective line bundles by a sufficiently high power of $p$ of a very ample line bundle (which is a $Q$-equivalence), we may assume that our line bundles are very ample.
In other words, $[Y\stackrel{y}{\row} X]$ is a $Q$-numerically trivial complete intersection, which stays numerically trivial over $E$.

\noindent \underline{{\bf Step 5}}:
Now $u=\prod_{i=1}^nc_1^{Q}(L_i)$, where $L_i$ are very ample line bundles and $u$ and $u_E$ are $Q$-numerically trivial. Let $L_i=O(D_i)$ and $\pp^{N_i}=|D_i|$ be the respective linear system.
Let $\pp=\prod_{i=1}^n\pp^{N_i}$ and $W$ be the intersection of generic representatives of our linear systems. It is defined over the purely
transcendental extension $k(\pp)$ of $k$. Then $u_{k(\pp)}=[W\row X]$
and $W$ is $Q$-anisotropic as long as $u$ is $Q$-numerically trivial
- see \cite[Proposition 4.5]{INCHKm}. In particular, $W$ is $Q$-anisotropic over $k(\pp)$ and $E(\pp)$. 

Thus, we have shown that, over some purely transcendental extension
$k(\pp)/k$, our original class $u$ will be equal to a push-forward from
some $Q$-anisotropic variety, which will remain anisotropic over $E$.
If $I\subset J\subset BP$ is another invariant ideal containing $I$ and 
$\ov{Q}^*=BP^*/J$, then, exactly as for $Q^*$, the above steps 
produce a variety whose $\ov{Q}$-anisotropy controls the $\ov{Q}$-numerical triviality of $\ov{u}\in\ov{Q}^*(X)$ (note that the only places where we care about particular shape of the theory and make some choices are the lifting to integral cobordism in Step 1 and the use of $Q$-equivalence in Steps 2 and 4;
in both cases, our choices for $Q^*$ will work for $\ov{Q}^*$).
 \Qed
\end{proof}

\begin{rem}
 \label{rem-nt-an}
The above Theorem shows that, for any theory $Q^*=BP^*/I$, with
$I$-invariant, the numerical triviality of a $Q^*$-class is equivalent to
the $Q$-anisotropy of a certain variety defined over the flexible closure
of the base field. 
 \Red
\end{rem}

\section{$p$-primary Chow groups vis $I(\infty)$-primary $BP$-theory}

For some distinct oriented cohomology theories, the notion of anisotropy is identical. For example, this happens for the theories $P(m)^*=BP^*/I(m)$ and $K(m)^*$ - see \cite[Proposition 4.16]{INCHKm}.
It is not difficult to see that the same holds for $\CH^*/p^r$ and 
$BP^*/I(\infty)^r$.

\begin{prop}
 \label{CHpr-BPIr}
{\rm(\cite[Proposition 2.1]{contr-IN})}
Let $X$ be smooth projective. Then TFAE
 \begin{itemize}
  \item[$(1)$] $X$ is $\CH/p^r$-anisotropic;
  \item[$(2)$] $X$ is $BP/I(\infty)^r$-anisotropic.
 \end{itemize}
\end{prop}

Since the ideal $I(\infty)^r$ is invariant under Landweber-Novikov operations (and non-zero), the {\it Isotropy Conjecture} holds for
$BP^*/I(\infty)^r$. But it doesn't hold for $\CH^*/p^r$, as was shown in
\cite{contr-IN}. In this sense, $BP^*/I(\infty)^r$ may be considered as
a "regular" substitute for $\CH^*/p^r$. 

It follows from Theorem \ref{nt-an} that $BP/I(\infty)^r$-numerical triviality of classes may be reduced to $BP/I(\infty)^r$-anisotropy of projective varieties which, due to Proposition \ref{CHpr-BPIr}, is equivalent to $\CH/p^r$-anisotropy. The following result shows that the $\CH/p^r$-anisotropy, in turn, may be reduced to the $\CH/p$-anisotropy.

Let $Y$ be a variety. Denote as $S^p(Y)$ its $p$-th symmetric power, and as
$\Lambda^p(Y)$ the complement in $S^p(Y)$ to the large diagonal. Note that
if $Y$ is smooth, then so is $\Lambda^p(Y)$.

\begin{prop}
 \label{sym-powers}
Let $Y$ be a variety. Then, for any $m\in\nn$, TFAE: 
 \begin{itemize}
  \item[$(1)$] $Y$ is $p^{m+1}$-anisotropic;
  \item[$(2)$] $\Lambda^p(Y)$ is $p^m$-anisotropic;
  \item[$(3)$] $S^p(Y)$ is $p^m$-anisotropic.
 \end{itemize}
\end{prop}

\begin{proof}
\noindent$(1\low 2)$
If $P$ is a point of degree $d$ on $Y$, then $\Lambda^p(Y)$ contains $\Lambda^p(P)$, which has degree $\binom{d}{p}$. So, if $p^{m+1}\not{|}d$, then $\Lambda^p(Y)$ has a point of degree not divisible by $p^m$. 

\noindent$(2\low 3)$ This is obvious, since $\Lambda^p(Y)$ is an open subscheme of $S^p(Y)$.

\noindent$(3\low 1)$
Let $Z=\{(y,x)\,|\,y\in x\}\subset Y\times S^p(Y)$ be the natural correspondence. It is flat of degree $p$ over $S^p(Y)$. Applying it to a point
of degree $d$ on $S^p(Y)$, we obtain a cycle of degree $p\cdot d$ on $Y$.
Thus, if $p^m\not{|}d$, then $Y$ has a point of degree not divisible by
$p^{m+1}$.
 \Qed
\end{proof}

Applying $(r-1)$ times the construction $\Lambda^p$ to the variety $X'$
from Theorem \ref{nt-an} and compactifying the result, we obtain a smooth
projective variety whose $p$-anisotropy controls the numerical triviality of the
original numerically trivial class in $BP^*(X)/I(\infty)^r$.

\begin{cor}
Let $X$ be a smooth projective variety and $u\in BP^*(X)/I(\infty)^r$ be a 
numerically trivial element. Then there exists a purely transcendetal field
extension $k(\pp)/k$ and a smooth projective variety $Y$ over $k(\pp)$, such that, for any
field extension $E/k$, $u_E$ is $BP/I(\infty)^r$-numerically trivial if and only if
$Y_{E(\pp)}$ is $p$-anisotropic.
\end{cor}

\section{The case of prime $2$}

In the case of $p=2$, we know from \cite[Theorem 1.1]{VPS} that the $2$-isotropy of projective varieties is controlled by pure symbols in $K^M_*/2$
of the flexible closure of the base field. So, we may reduce the numerical triviality
of cycles from $BP^*/I(\infty)^r$ (for $p=2$) to the non-triviality of certain pure symbols. The following result permits to do it in a coherent way, for different $r$.

\begin{prop}
 \label{alpha-beta}
Let $Y\row Z$ be a closed embedding of projective varieties. Suppose, $\alpha\in K^M_*(\wt{k})/2$ is a pure symbol controlling the isotropy of $Y$. Then there
exists a pure symbol $\beta\in K^M_*(\wt{k})/2$ controlling the isotropy of $Z$,
such that $\alpha|\beta$.
\end{prop}

\begin{proof}
Suppose, $\alpha$ is defined over some purely transcendental extension $k'/k$
and $Q_{\alpha}$ be the respective Pfister quadric (defined over $k'$).
Then, for any field extension $F/k$ and $F'=F*_k k'$, the $2$-isotropy of
$Y_F$ is equivalent to that of $Q_{\alpha}|_{F'}$. Hence, the $2$-isotropy of
$Z_F$ is equivalent to that of $(Z_{k'}\coprod Q_{\alpha})_{F'}$. By \cite[Theorem 1.1]{VPS}, there exists a purely transcendental extension
$k''/k'$ and a pure symbol $\beta\in K^M_*(k'')/2$, such that, for any field extension $F'/k'$ and $F''=F'*_{k'}k''$, the $2$-isotropy of  
$Z_{k'}\coprod Q_{\alpha}$ is equivalent to the triviality of $\beta_{F''}$.
In particular, $\beta$ vanishes over $k''(Q_{\alpha})$, so, is divisible by $\alpha$.
And the $2$-isotropy of $Z$ is controlled by $\beta$.
 \Qed
\end{proof}

Now we can describe the numerical triviality of $BP^*$ classes modulo various powers of the augmentation ideal $I(\infty)$ in terms of non-triviality of pure symbols over the flexible closure.

\begin{thm}
 \label{ntBP-ps}
Let $p=2$, $X$ be a smooth projective variety over $k$ and $u\in BP^*(X)$.
Then there exists a chain of pure symbols $\alpha_r\in K^M_*(\wt{k})/2$
over the flexible closure of the base field, so that, 
$\alpha_1|\alpha_2|\alpha_3|\ldots$ and, for any field extension $F/k$,
$u_F$ is $BP/I(\infty)^r$-numerically trivial if and only if $\alpha_r|_{\wt{F}}\neq 0$.
\end{thm}

\begin{proof}
If $u$ is not $BP/I(\infty)^r$-numerically trivial, for any $r$, we may set: 
$\alpha_r=0$, for any $r$. If $u$ is $BP/I(\infty)^r$-numerically trivial, for every $r$, then it is $BP$-numerically trivial, and this property will stay over arbitrary field extensions (by transfer arguments), so we may set: $\alpha_r=1=\{\}$, for all $r$. So, we may assume that there exists $n\in\nn$, such that $u$ is
$BP/I(\infty)^r$-numerically trivial, for $r\leq n$, and is not numerically trivial,
for $r>n$.

Let $Q^*=BP^*/I(\infty)^n$ and $Y$ be the variety from Theorem \ref{nt-an}
which controls the $Q$-numerical triviality of $u$. Then it will also control
the $BP/I(\infty)^r$-numerical triviality of it, for $r\leq n$, since 
$I(\infty)^n\subset I(\infty)^r\subset BP$. 
More precisely, the $BP/I(\infty)^r$-numerical triviality of $u$ is equivalent to the
$BP/I(\infty)^r$-anisotropy of $Y$, or, which is the same, the $\CH/2^r$-anisotropy of it.
We have closed embeddings of projective varieties
$$
Y\row S^2(Y)\row S^2(S^2(Y))\row\ldots,
$$
where the $2$-anisotropy of $S^2(\ldots (S^2(Y)))$ ($i$-times) is equivalent to
the $BP/I(\infty)^{i+1}$-numerical triviality of $u$, by 
Proposition \ref{sym-powers}. 
Applying the Proposition \ref{alpha-beta} inductively, we get the chain of pure symbols $\alpha_1|\alpha_2|\ldots|\alpha_n\in K^M_*(\wt{k})/2$ controlling the $BP/I(\infty)^r$-numerical triviality of $u$, for $1\leq r\leq n$. It remains to set
$\alpha_r=0$, for $r>n$.
 \Qed
\end{proof}

\begin{rem}
The same holds for $u\in BP^*_{I(\infty)}=\operatornamewithlimits{lim}_rBP^*/I(\infty)^r$ - the completion of the $BP^*$-theory at the ideal $I(\infty)$.
 \Red
\end{rem}

\end{document}